\documentclass[11pt]{amsart}
\usepackage{amscd, amssymb}
\usepackage{graphics}

\theoremstyle{plain}

\numberwithin{equation}{section}

\newtheorem{theorem}[equation]{Theorem}
\newtheorem{proposition}[equation]{Proposition}
\Huge
\newtheorem{lemma}[equation]{Lemma}

\theoremstyle{definition}
\newtheorem{remark}[equation]{Remark}
\newtheorem{example}[equation]{Example}

\newtheorem*{assump A}{Assumption A}
\newtheorem*{assump B}{Assumption B}

\def    \R  {{\Bbb R}}
\def    \Z  {{\Bbb Z}}
\def    \Q  {{\Bbb Q}}

\def    \CP {{\Bbb {CP}}}
\def    \P {{\Bbb {P}}}
\def    \C  {{\Bbb C}}
\def    \N  {{\Bbb N}}

\def  \deg {{\operatorname{deg}}}

\def    \Tilde  {\widetilde}
\def    \ut     {\Tilde{u}}

\def    \Gt     {\Tilde{G}}

\begin{document}
\title[Certain circle actions on K\"ahler manifolds] {Certain circle actions on
 K\"ahler manifolds}

\author{Hui Li}
\address{School of mathematical Sciences, Box 173\\
               Suzhou University \\
               Suzhou, 215006, China.}
        \email{hui.li@suda.edu.cn}

\thanks{2000 classification. 53D05, 53D20, 55N25, 57R20, 32H02, 14J70, 53D50} 
\begin{abstract}
Let the circle act holomorphically on a
compact K\"ahler manifold $M$ of complex dimension $n$ with moment map
$\phi\colon M\to\R$. Assume the critical set of $\phi$ consists of $3$ connected components, the extrema being isolated points. We show that $M$ is equivariantly biholomorphic to $\CP^n$, where $n\geq 2$, or to $\Tilde G_2(\R^{n+2})$, the Grassmannian of oriented $2$-planes in $\R^{n+2}$, where $n\geq 3$, with a standard circle action; we also show that $M$ is equivariantly symplectomorphic to $\CP^n$, where $n\geq 2$, or to $\Tilde G_2(\R^{n+2})$, where $n\geq 3$, with a standard circle action.
\end{abstract}

 \maketitle

\section{Introduction}

Let $(M,\omega)$ be a connected compact symplectic manifold.
Assume that $M$ admits a nontrivial\footnote{We always assume the
action is nontrivial unless otherwise stated.} Hamiltonian circle
action with moment map $\phi\colon M\to \R$. We call $M$ a Hamiltonian $S^1$-manifold.
The moment map $\phi\colon M \to \R$ is a Morse-Bott
function whose critical set is exactly the fixed point set $M^{S^1}$ of the circle action.

Let $(M, \omega)$ be a compact Hamiltonian $S^1$-manifold whose even Betti numbers are minimal, i.e., $b_{2i}(M) = 1$ for all $0\leq 2i\leq \dim(M)$; or, whose even Betti numbers are the same as those
of a complex projective space. Some recent work studied certain cases of such Hamiltonian $S^1$-manifolds. The papers \cite{{T}, {M}} studied the case when $\dim(M)=6$, \cite{GS} studied the case when $\dim(M) = 8$ and $M^{S^1}$ consists of isolated points, and \cite{{LT}, {LOS}} studied the case when $M$ is of any dimension and $M^{S^1}$ consists of two connected components.  These papers showed that such Hamiltonian $S^1$-manifolds have certain geometrical and topological invariants the same as those of the complex projective spaces, the Grassmannian of oriented two planes in some $\R^N$, or some other relatively well known K\"ahler manifolds. In particular, for two important and interesting cases when $\dim(M) = 6$ and when the fixed points are isolated, \cite{M} determined the equivariant symplectomorphism type of the manifolds. While for the case when $M$ is $2n$-dimensional and $M^{S^1}$ consists of two connected components, \cite{LOS} showed that there are finitely many such manifolds up to equivariant diffeomorphism; only in a few special cases, we know the equivariant symplectomorphism type of the manifold, in particular in the case when one fixed component is isolated, by a theorem of Delzant \cite{D}, $M$ is equivariantly symplectomorphic to $\CP^n$ with a standard circle action.

Let $(M, \omega)$ be a compact Hamiltonian $S^1$-manifold. In \cite[Section 4]{LT}, it was shown that if $b_{2i}(M) = 1$ for all $0\leq 2i\leq \dim(M)$, then the fixed components of the circle action satisfy
\begin{equation}\label{equality}
\sum_{F \subset M^{S^1} } \left( \dim(F)  + 2 \right) = \dim(M) + 2,
\end{equation}
where the sum is over all the fixed set components. While the converse is false as shown by Example~\ref{ex2} below (when $n$ is even),
it is anticipated that when (\ref{equality}) holds, the cohomology groups of $M$ are relatively ``simple".

\smallskip

In this paper,  we consider a K\"ahler manifold satisfying the following assumption:
\begin{assump A}
Let $(M, \omega, J)$ be a compact K\"ahler manifold of complex dimension $n$ which admits a holomorphic Hamiltonian circle action. Assume the critical set of the moment map consists of $3$ connected components, the extrema being isolated points.
\end{assump A}

Under this assumption, by Morse theory (see Lemma~\ref{1dim}), we can see that the $3$ critical components of the moment map satisfy the equality (\ref{equality}). Hence this is a special case of a Hamiltonian $S^1$-manifold with fixed point set satisfying (\ref{equality}). We have two families of examples satisfying Assumption A.

 \begin{example}\label{ex1}
Consider $\CP^n$, where $n\geq 2$. As a coadjoint orbit of $SU(n+1)$, it admits
a K\"ahler structure. Consider the effective $S^1\subset
SU(n+1)$ action on $\CP^n$ given by
$$\lambda\cdot [z_0, z_1, \cdots, z_n] = [\lambda^{-l} z_0, \lambda^{l'} z_1, z_2, \cdots, z_n],$$
where $l, l'\in\N$, and $\gcd(l, l') = 1$.  The fixed point set $(\CP^n)^{S^1}$ consists of $3$
connected components:
$$X = [z_0, 0, \cdots, 0], \,\, Y = [0, 0, z_2, \cdots, z_n],\,\,\mbox{and}\,\,
Z = [0, z_1, 0, \cdots, 0],$$
where $X$ and $Z$ are isolated points. The action is holomorphic and is Hamiltonian.
The weights of the circle action on the normal bundles of $X$, $Y$ and $Z$ are respectively
 $$\big(l, \cdots, l, l+l'\big), \,\, (-l, l'),\,\, \big(-l', \cdots, -l', -(l+l')\big).$$
\end{example}

\begin{example}\label{ex2}
Consider $\Tilde G_2(\R^{n+2})$, the Grassmannian of oriented $2$-planes in $\R^{n+2}$,
where $n\geq 3$. We have $\dim_{\R}\big(\Tilde G_2(\R^{n+2})\big)=2n$. As a coadjoint orbit of $SO(n+2)$, it has a K\"ahler structure. Consider the $S^1\subset SO(n+2)$ action on
$\Tilde G_2(\R^{n+2})$ induced by the action on $\R^{n+2} =
\C\times\R^n$ given by
$$\lambda\cdot \left(z, x_1, \dots, x_n\right)=
 \left(\lambda\cdot z, x_1, \dots, x_n\right).$$
The fixed point set $\big(\Tilde G_2(\R^{n+2})\big)^{S^1}$ consists of $3$ connected components:
 $$X = \P\left(z, 0, \cdots, 0\right) =\mbox{pt}, \, \, Y = \Gt_2\big(0\times \R^n\big), \,\,\mbox{and}\,\,
 Z = \P\left(z, 0, \cdots, 0\right) = \mbox{pt},$$
 where $X$ and $Z$ correspond to the two orientations on the real $2$-plane $(z, 0, \cdots, 0)$.
 The $S^1$ action is holomorphic and Hamiltonian.
 The weights of the circle action on the normal bundles of $X$, $Y$ and $Z$ are respectively
 $$(1, \cdots,  1), \,\, (-1, 1), \,\, (-1, \cdots,  -1).$$
\end{example}

Our results in Theorems~\ref{biho} and \ref{symp}
equivariantly identify the manifold satisfying Assumption A with one of these examples in the complex and symplectic categories.

\begin{theorem}\label{biho}
Under Assumption A, $M$ is $S^1$-equivariantly biholomorphic to $\CP^n$, where $n\geq 2$, or to $\Tilde G_2(\R^{n+2})$, where $n\geq 3$, with a standard circle action.
\end{theorem}

\begin{theorem}\label{symp}
Under Assumption A, $M$ is $S^1$-equivariantly symplectomorphic to $\CP^n$, where $n\geq 2$, or to $\Tilde G_2(\R^{n+2})$, where $n\geq 3$, with a standard circle action.
\end{theorem}

Under the following assumption, where $M$ is a symplectic manifold,
our next result is that the circle action and the first Chern class of the manifold are exactly as those of an example in Example~\ref{ex1} or \ref{ex2}.
\begin{assump B}
Let $(M, \omega)$ be a compact
$2n$-dimensional symplectic manifold which admits an effective Hamiltonian circle action. Assume the critical set of the moment map $\phi$ consists of $3$ connected components $X$, $Y$ and $Z$, where $\phi(X) < \phi(Y) < \phi(Z)$, and $X$ and $Z$ are isolated points.
\end{assump B}

\begin{theorem}\label{wc1}
Suppose Assumption B holds. Then $n$ and the weights of the $S^1$ action respectively on the normal bundles of $X$, $Y$ and $Z$ are
\begin{enumerate}
\item  $n\geq 3; \quad  (1, \cdots, 1),\,\, (-1, 1),\,\,  (-1, \cdots, -1),$ or
\item $n\geq 2; \quad  (l, \cdots, l, l + l'),\,\, (-l, l'), \,\, \left(-l', \cdots, -l',
-(l + l')\right),$  where $l, l'\in \N$, and $\gcd(l, l') = 1$.
\end{enumerate}
Moreover, if $[\omega]$ is a primitive integral class, then
$c_1(M) = n [\omega]$  in case $(1)$, and
$ c_1(M) = (n+1) [\omega]$ in case $(2)$.
\end{theorem}

We use equivariant cohomology and Morse theory to prove Theorem~\ref{wc1}. Then we use the K\"ahler condition and a result
by Kobayashi and Ochiai to prove Proposition~\ref{equibiho}, and we use Theorem~\ref{wc1} and Proposition~\ref{equibiho} to prove Theorem~\ref{biho}. Finally, we use Theorem~\ref{biho} to prove Theorem~\ref{symp}.

The work of this paper provides a new idea to study Hamiltonian
 $S^1$-manifolds which are K\"ahler and which have fixed point set satisfying (\ref{equality}).
We hope to be able to use the method to treat more general cases of the fixed point set for compact K\"ahler  Hamiltonian $S^1$-manifolds.

For the case when $(M, \omega)$ is a compact $2n$-dimensional Hamiltonian $S^1$-manifold such that the critical set of the moment map consists of two connected components, $X$ and $Y$,  \cite{LT} showed that the condition  $b_{2i}(M) = 1$ for all $0\leq 2i\leq \dim(M)$ is equivalent to the condition $\dim(X) + \dim(Y) + 2 = \dim(M)$.  Under this assumption, \cite[Theorem 1]{LT} showed that, $c_1(M) = (n+1)[\omega]$, or $c_1(M) = n[\omega]$ with $n\geq 3$ odd, if $[\omega]$ is primitive integral.
If we assume in addition that $M$ is K\"ahler and the $S^1$ action is holomorphic,  then  our method (see Proposition~\ref{equibiho} and our proofs of Theorems~\ref{biho}  and \ref{symp})  implies that $M$ is $S^1$-equivariantly biholomorphic to $\CP^n$, or to $\Tilde G_2(\R^{n+2})$ with
$n\geq 3$ odd, with a standard circle action; and $M$ is $S^1$-equivariantly symplectomorphic to $\CP^n$, or to $\Tilde G_2(\R^{n+2})$ with
$n\geq 3$ odd, with a standard circle action.

\smallskip

A compact  K\"ahler manifold $M$ with the same Betti numbers as $\CP^n$ but is different from $\CP^n$ is called a {\it fake projective space}. For example, $\Tilde G_2(\R^{n+2})$ with $n\geq 3$ odd are such spaces. There are more fake projective spaces other than these, see \cite{PY, IS, PY1} and etc. for the classification and study in complex dimensions $2, 3, 4$. But with stronger topological conditions on a  compact K\"ahler manifold $M$,  one can exclude the fake ones, or show that $M$ is biholomorphic to $\CP^n$: if $M$ has the homotopy type and Pontryagin classes of $\CP^n$ \cite{HK}; if $M$ has the homotopy type of $\CP^n$ and $n\leq 6$ \cite{LW};  or if $M$ has positive first Chern class, and its integral cohomology ring is isomorphic to that of $\CP^n$ with $n\leq 5$  \cite{F}.

Both in the symplectic and K\"ahler cases,  the study of compact Hamitonian $S^1$-manifolds with minimal even Betti numbers, or with the weaker condition -- the fixed point set satisfying (\ref{equality}),  provides a different perspective to the extensive and interesting study of rational cohomology complex projective spaces.

\subsubsection*{Acknowledgement} I thank M. De Cataldo for a discussion on projective
 manifolds. I especially thank the referee for a few helpful suggestions, in particular for a suggestion which led to  the simplified proof of Theorem~\ref{symp}.

\section{Some preliminary results}
In this section, we state and prove some preliminary results for symplectic Hamiltonian $S^1$-manifolds --- for the general case and for the case when the Hamiltonian function has $3$ critical components, with isolated extrema.

 Let $(M, \omega)$ be a Hamiltonian $S^1$-manifold with moment map $\phi\colon M\to\R$, and let
 $F$ be a fixed component of the $S^1$ action. Let us set up the following notations.
\begin{itemize}
\item $2\lambda_F$: the Morse index of $F$ as a critical set of the Morse-Bott function $\phi$;
\item $2\lambda_F^+$: the Morse coindex of $F$ for $\phi$, or the Morse index of $F$ for $-\phi$;
\item $\Gamma_{F}$: the sum of the weights of the $S^1$ action on the normal bundle $N_F$ to $F$;
\item $\Lambda_{F}$:  the product of the weights of the $S^1$ action on the normal bundle $N_F$ to $F$;
\item $\Lambda_{F}^-$:  the product of the weights of the $S^1$ action on the negative normal bundle $N_F^-$ to $F$;
\item $\Lambda_{F}^+$:  the product of the weights of the $S^1$ action on the positive normal bundle $N_F^+$ to $F$.
\end{itemize}

 For a smooth $S^1$-manifold $M$, let $H^*_{S^1}\left(M; R\right) = H^*\left( S^{\infty}\times_{S^1} M; R\right)$ be the $S^1$-equivariant cohomology of the manifold $M$ in $R$ coefficient, where $R$ is a
 coefficient ring. Let $t\in H^2_{S^1}\left(\mbox{pt};  \Z\right) = H^2\left( \CP^{\infty};  \Z\right)$ be a generator.

 The projection $\pi\colon S^{\infty}\times_{S^1} M \to\CP^{\infty}$ induces a natural push forward map $\pi_*\colon H^*_{S^1}(M;\Q)\to H^*(\CP^{\infty};\Q)$, which is given by ``integration over the fiber", denoted
$\int_M$. We will use the following theorem due to Atiyah-Bott, and Berline-Vergne \cite{AB, BV} in Section~\ref{sec.wc1}.
\begin{theorem}\label{AB.BV}
Let the circle act on a compact manifold $M$. Fix a class $\alpha\in H^*_{S^1}(M; \Q)$. Then as elements of $\Q(t)$,
$$\int_M \alpha = \sum_{F\subset M^{S^1}}\int_F \frac{\alpha|_F}{e^{S^1}(N_F)},$$
where the sum is over all fixed components, and $e^{S^1}(N_F)$ is the equivariant Euler class of the normal bundle to $F$.
\end{theorem}

\subsection{Some preliminary results on Hamiltonian $S^1$-manifolds}

\
\medskip

First, we have the following ``equivariant extension" of the cohomology class represented by the symplectic
form.

\begin{lemma}\cite[Lemma 2.7]{LT}\label{extension}
Let the circle act  on a connected compact symplectic manifold $(M,\omega)$
with moment map $\phi \colon M \to \R$.  Let $F$ be a fixed  component.  Then there exists $\ut \in H_{S^1}^2(M;\R)$ so that
$$\ut|_{F'} = [\omega|_{F'}] + t \left(\phi(F) - \phi(F') \right)$$
for any fixed component $F'$. If $[\omega]$ is an integral class,
then  $\ut$ is an integral class.
\end{lemma}

Now, for a Hamiltonian $S^1$-manifold $M$ with $H^2(M; \R) = \R$, we use the above lemma to express the 1st Chern class $c_1(M)$ of $M$ in terms of data related to any two fixed components of the action.

\begin{lemma}\label{c1}
Let the circle act on a connected compact symplectic manifold
$(M, \omega)$ with moment map $\phi\colon M\to\R$. Assume $H^2(M; \R) = \R$.
 Then
$$c_1(M) = \frac{ \Gamma_{F} - \Gamma_{F'}}{\phi(F') - \phi(F)} [\omega],$$
where $F$ and $F'$ are any two fixed components.
\end{lemma}

\begin{proof}
Since  $H^2(M; \R) = \R$, $[\omega]\in H^2(M; \R)$ is the generator. So
$t$ and the $\ut$ in Lemma~\ref{extension} are the generators of $H^2_{S^1}(M; \R)$.
We can write the equivariant first Chern Class $c_1^{S^1}(M)$ of $M$ as
$$c_1^{S^1}(M) = a t + b \ut, \,\,\mbox{where $a, b \in\R$}.$$
Then
$$c_1^{S^1}(M)|_f = \Gamma_F t = a t, \,\,\mbox{where $f\in F$ is a point, and}$$
$$c_1^{S^1}(M)|_ {f'} =\Gamma_{F'} t = at + bt\left(\phi(F) - \phi(F')\right), \,\,\mbox{where $f'\in F'$ is a point}.$$
So
$$a = \Gamma_F, \quad  b = (\Gamma_{F'} - \Gamma_F)/\left(\phi(F) - \phi(F')\right).$$
Hence
$$c_1^{S^1}(M) =  \Gamma_F t + \frac{ \Gamma_{F'} - \Gamma_F}{\phi(F) - \phi(F')}\ut.$$
Taking the restriction map $H^2_{S^1}(M; \R)\to H^2(M; \R)$, we get
$$c_1(M) = \frac{ \Gamma_{F'} - \Gamma_F}{\phi(F) - \phi(F')} [\omega].$$
\end{proof}

 The following lemma gives an expression of the equivariant Euler class of the negative normal bundle of
 a fixed component $F$ when its Morse index is the sum $\sum_{\phi(F') < \phi(F)}\left(\dim(F')+2\right)$.

\begin{lemma}\cite[Remark 4.12]{LT}\label{Euler}
Let the circle act on a connected compact  symplectic manifold $(M,
\omega)$ with moment map $\phi\colon M\to \R$.
Let $F$ be a fixed component and assume $\sum_{\phi(F') < \phi(F)}\left(\dim(F')+2\right)=2\lambda_F$.
Then
$$e^{S^1}(N_F^-)=\Lambda_F^- \prod_{\phi(F') < \phi(F)}\left( t + \frac{[\omega|_F]}{\phi(F')-\phi(F)}\right)^{\frac{1}{2}\dim(F')+1},$$
 where $e^{S^1}(N_F^-)$ is the equivariant Euler class of the
negative normal bundle  $N_F^-$ of $F$.
 \end{lemma}

\subsection{The case when the Hamiltonian function has $3$ critical components, the extrema being isolated}

\
\medskip

 First we show that, the index of the non-extremal critical component of the moment map $\phi$ satisfies the condition
 of Lemma~\ref{Euler}, and $H^2(M; \R)\cong \R$.

\begin{lemma}\label{1dim}
Let the circle act on a compact  symplectic manifold $(M,
\omega)$ with moment map $\phi\colon M\to \R$ such that its critical set consists of $3$ connected components $X$, $Y$ and $Z$,
where $\phi(X) < \phi(Y) < \phi(Z)$, with $X$ and $Z$ being isolated points. Then $2\lambda_Y = 2\lambda_Y^+ = 2$, and $H^2(M; \R) \cong \R$.
\end{lemma}

\begin{proof}
Since $\phi$ has $3$ critical components, we have $\dim(M) > 2$. Since $Z$ is isolated, we have $2\lambda_Z = \dim(M) > 2$.
The fact that $\phi$ is a perfect Morse-Bott function gives
\begin{equation}\label{pmb}
\dim H^i(M) = \sum_{F\subset M^{S^1}} \dim H^{i-2\lambda_F}(F).
\end{equation}
Then since $X$ is isolated, and $H^2(M; \R)\neq 0$, (\ref{pmb}) gives
$2\lambda_Y = 2$. Similarly, using $-\phi$, we get $2\lambda_Y^+ = 2$.
Now the fact $H^2(M; \R) \cong \R$ follows from these facts and (\ref{pmb}).
\end{proof}

 Next, for the case stated, we write $c_1(M)$ slightly differently from that in Lemma~\ref{c1} so that we
 can identify an integral class.

\begin{lemma}\label{c1gammaf}
Let the circle act on a compact  symplectic manifold $(M,
\omega)$ with moment map $\phi\colon M\to \R$ such that its critical set consists of $3$ connected components $X$, $Y$ and $Z$,
where $\phi(X) < \phi(Y) < \phi(Z)$, with $X$ and $Z$ being isolated points. Then
$$c_1(M) = \frac{\Gamma_X -\Gamma_Y}{\Lambda_Y^-} \left[\Lambda_Y^-
\,\frac{\omega}{\phi(Y)-\phi(X)}\right],$$
where $[\Lambda_Y^- \,\frac{\omega}{\phi(Y) -\phi(X)}]$ is an integral class.
\end{lemma}

\begin{proof}
By Lemma~\ref{extension}, we can take $\ut\in H^2_{S^1}(M; \R)$ so that $\ut|_Y = [\omega|_Y]$ and $\ut|_X = t \big(\phi(Y) - \phi(X)\big)$. Using this together with  Lemmas~\ref{1dim} and \ref{Euler}, we get
\begin{equation}\label{restrY}
e^{S^1}\left(N_Y^-\right) = \Lambda_Y^- \left( t + \frac{[\omega|_Y]}{\phi(X) - \phi(Y)}\right) = \Lambda_Y^- \left( t + \frac{\ut|_Y}{\phi(X) - \phi(Y)}\right).
\end{equation}
 Since $M^{S^1}$ consists of 3 connected components and $Z$ is isolated, $\dim(M)
= 2\lambda_Z \geq 4$, so $\deg \big(e^{S^1}\left(N_Y^-\right)\big) =  2 < 2\lambda_Z$.
Let $M^- = \{x\in M\,|\, \phi(x) < \phi(Z)\}$. Consider the
long exact sequence for the pair $(M, M^-)$ in equivariant cohomology with $\Z$ coefficients, since $H^2_{S^1}(M, M^-; \Z) =H^{2 - 2\lambda_Z }_{S^1}(Z; \Z) = 0$ and $H^3_{S^1}(M, M^-; \Z) = H^{3 - 2\lambda_Z }_{S^1}(Z; \Z) = 0$, the long
exact sequence splits into the short exact sequence
\begin{equation}\label{short}
0\to H^2_{S^1}(M; \Z)\to H^2_{S^1}(M^-; \Z)\to 0.
\end{equation}
Consider the class $\Tilde {e^{S^1}\left(N_Y^-\right)} = \Lambda_Y^- \left( t + \frac{\ut}{\phi(X) - \phi(Y)}\right)\in H^2_{S^1}(M; \R)$. By (\ref{restrY}), its restriction to $Y$ is
$e^{S^1}\left(N_Y^-\right)$; clearly, its restriction to $X$ is $0$.
By (\ref{short}), $\Tilde {e^{S^1}\left(N_Y^-\right)}$ is the unique class on $M$ having this property.
So $e^{S^1}\left(N_Y^-\right)$ is integral implies $\Tilde {e^{S^1}\left(N_Y^-\right)}$ is integral. Taking the restriction map $H^*_{S^1}(M; \R)\to H^*(M; \R)$ for the class $\Tilde {e^{S^1}\left(N_Y^-\right)}$,  we get that $[\Lambda_Y^- \,\frac{\omega}{\phi(X) - \phi(Y)}]$ is an integral class.

 By Lemma~\ref{1dim}, $H^2(M; \R) = \R$. Now apply Lemma~\ref{c1} using the fixed components $X$ and $Y$.
\end{proof}

\section{Proof of Theorem~\ref{wc1}}\label{sec.wc1}

Let $(M, \omega)$ be a symplectic $S^1$-manifold. An {\it isotropy submanifold} $M^{\Z_k}\subsetneq M$ is a symplectic submanifold which is not fixed by the $S^1$ action, but is fixed by the $\Z_k$ action for some $k > 1$.  We need the following lemma in the proof of Theorem~\ref{wc1}.

\begin{lemma}\cite[Lemma 2.6]{T}\label{modulo}
Let the circle act symplectically on a compact symplectic manifold $(M, \omega)$.
Let $p$ and $q$ be fixed points which lie on the same component $N$
of $M^{\mathbb Z_k}$ for some $k>1$. Then the weights of the action
at $p$ and at $q$ are equal modulo $k$.
\end{lemma}

\begin{proof}[Proof of Theorem~\ref{wc1}]
By Lemma~\ref{1dim}, $2\lambda_Y = 2\lambda_Y^+ = 2$. We have $2\lambda_Z = 2n \geq 4$.

{\bf Case (1).} Assume the action is semifree.  Then the weights of the $S^1$ action on the normal bundles of $X$, $Y$ and $Z$ are as in $(1)$ of Theorem~\ref{wc1}.

   Suppose $Y$ is an isolated point, then $\dim(M) = \dim(Y)+2\lambda_Y +2\lambda_Y^+ = 4$,  $e^{S^1}(N_X)= t^2$,  $e^{S^1}(N_Y)= t\cdot (-t)$ and $e^{S^1}(N_Z)= (-t)^2$. Using Theorem~\ref{AB.BV} to integrate $1$ on $M$, we get a contradiction. Hence $\dim(M)\geq 6$.

    Using Lemma~\ref{c1gammaf}, we get
$$c_1(M) = n \left[\frac{\omega}{\phi(Y)
-\phi(X)}\right], \,\,\mbox{where $[\frac{\omega}{\phi(Y) -\phi(X)}]\in H^{2}(M; \Z)$}.$$
If $[\omega]$ is primitive integral, then $\phi(Y) -\phi(X)\in\N$ by Lemma~\ref{extension}; so $\left[\frac{\omega}{\phi(Y)
-\phi(X)}\right] = [\omega]$ and $c_1(M) = n[\omega]$.

{\bf Case (2).} Assume the action is effective and is not semifree.

Case (2a). Assume there is an isotropy submanifold $M_X^Y$ between $X$ and $Y$, which is an $M^{\Z_l}$, where $l> 1$, and there is an isotropy submanifold $M_Y^Z$ between $Y$ and $Z$, which is an $M^{\Z_{l'}}$, where $l' > 1$. We have $\dim(M_X^Y) = \dim(Y) + 2\lambda_Y = \dim(Y) + 2 + \dim(X)$,
and similarly $\dim(M_Y^Z) = \dim(Y) + 2 + \dim(Z)$; moreover, the weights of the normal representation at $Y$ are $(-l, l')$, where $\gcd(l, l') = 1$ since the action is effective. Then the weights of the normal representation at $X$ are $(l, \cdots, l, s)$ for some $s\in\N$. The weights of the normal representation at $Z$ must be $(-l', \cdots, -l', -s)$.

By Lemma~\ref{1dim}, $H^2(M; \R) \cong \R$. By rescaling, we may assume that $[\omega]$ is primitive integral.
Since $2 < 2\lambda_Z$, $[\omega|_{\{m\in M\,|\, \phi(m)<\phi(Z)\}}]$ is primitive integral. Since
the manifold $\{m\in M\,|\, \phi(m)<\phi(Z)\}$ deformation retracts to $M_X^Y$ (the retract is given by the gradient flow of $-\phi$), $[\omega|_{M_X^Y}]$ is primitive integral.
 Similarly, since $ 2 < 2\lambda_X^+$,  $[\omega|_{M_Y^Z}]$ is primitive integral.
 Using \cite[Proposition 6.1]{LT} on $M_X^Y$ and on $M_Y^Z$, we get
 $$\phi(Y)-\phi(X) = l \quad\mbox{and}\quad \phi(Z)-\phi(Y) = l'.$$
 Then using Lemma~\ref{c1} to compute $c_1(M)$ respectively in terms of $X, Y$ and $Y, Z$,
 we get $$s = l + l' \quad\mbox{and}\quad  c_1(M) = (n+1) [\omega].$$

Case (2b). Assume there is an isotropy submanifold $M_X^Y$ between $X$ and $Y$, which is an $M^{\Z_l}$, where $l > 1$, and there is no isotropy submanifold between $Y$ and $Z$. Then the weights of the normal representation at $Y$ are $(-l, 1)$, and the weights of the normal representation at $X$ are $(l, \cdots, l, s)$ for some $s\in\N$. Similarly as in Case (2a), assume  $[\omega]$ is primitive integral, then
  $$\phi(Y)-\phi(X) = l.$$
First assume there is no isotropy submanifold between $X$ and $Z$. Then $s=1$
and the weights of the normal representation at $Z$ are $(-1, \cdots, -1)$. Using Lemma~\ref{c1}
to compute $c_1(M)$ respectively in terms of $X, Y$ and $Y, Z$, it gives a contradiction. Hence there is an isotropy submanifold between $X$ and $Z$, and there is only one, which corresponds to the weight $s$ (hence $s > 1$). So the weights of the normal representation at $Z$ are $(-1, \cdots, -1, -s)$.
Using Lemma~\ref{modulo} for the fixed sets $X$ and $Z$, we have  $l + 1 = a s$, where $a > 0$. Similarly, using the lemma
for the fixed sets $X$ and $Y$, we get $s=
1 + bl$, where $b > 0$. These together give $a=1$ and $b=1$, so
$$s = l +1.$$
Using Lemma~\ref{c1} to compute $c_1(M)$, we get
$$c_1(M) = (n+1) [\omega] \quad\mbox{and}\quad \phi(Z)-\phi(Y) = 1.$$

  For the case when there is an isotropy submanifold between $Y$ and $Z$ and there is no isotropy submanifold
  between $X$ and $Y$, we can discuss similarly and we have similar conclusion as above.

Case (2c). Assume there are no isotropy submanifolds between $X$ and $Y$ and between $Y$ and $Z$. Then the weights of the normal representation at $Y$ are $(-1, 1)$. Since the action is not semifree, there
is an isotropy submanifold between $X$ and $Z$. We will show that there is only one isotropy submanifold
between $X$ and $Z$, which is an $M^{\Z_2}$.  If $\dim(M) = 4$, using Lemma~\ref{modulo}, this is clear.

  Now assume $\dim(M) > 4$. Since $X$ and $Z$ are isolated points, each isotropy submanifold between
$X$ and $Z$ is a sphere. If $W$ is the set of weights at $X$, then the set of weights at $Z$ is $-W$. Then $\Gamma_X = - \Gamma_Z$ and $\Lambda_X = \pm \Lambda_Z$. We also have $\Lambda_Y^- = -1$, $\Lambda_Y^+ = 1$, and $\Gamma_Y = 0$. Assume $[\omega]$ is primitive integral. Then by Lemmas~\ref{c1gammaf} and \ref{extension} (using a similar argument as in {\bf Case (1)}),
\begin{equation}\label{c1m'}
c_1(M) = \Gamma_X\,[\omega],\,\,\mbox{and}\,\, \phi(Y) - \phi(X) = 1.
\end{equation}
So $c_1(M)|_Y = \Gamma_X\,[\omega|_Y]$. By symmetry or using a similar argument as
the above, $\phi(Z) - \phi(Y) = \phi(Y) - \phi(X)=1$.
By Lemmas~\ref{1dim} and \ref{Euler},
$e^{S^1}(N_Y^-) = -t + \frac{[\omega|_Y]}{\phi(Y) - \phi(X)} = -t + [\omega|_Y]$. Using these lemmas for $-\phi$, we get
$e^{S^1}(N_Y^+) = t + \frac{[\omega|_Y]}{\phi(Z) - \phi(Y)} = t + [\omega|_Y]$.
First assume $\dim(Y) = 2(2k+1)$, where $k\geq 0$.
Let $c_1^{S^1}(M)$ be the equivariant first Chern class of $M$. Then $c_1^{S^1}(M)|_X =\Gamma_X t$,
$c_1^{S^1}(M)|_Y = c_1(M)|_Y + t -t =\Gamma_X \,[\omega|_Y]$, $c_1^{S^1}(M)|_Z =\Gamma_Z t$.
We have $e^{S^1}(N_X)=\Lambda_X t^n$, $e^{S^1}(N_Y) =e^{S^1}(N_Y^-)\cdot e^{S^1}(N_Y^+)$ and $e^{S^1}(N_Z)=\Lambda_Z t^n$. Using Theorem~\ref{AB.BV} to compute
$$\int_M c_1^{S^1}(M) = 0,$$
we get
$$\frac{\Gamma_X}{\Lambda_X} - \Gamma_X a + \frac{\Gamma_Z}{\Lambda_Z} = 0,\,\,\mbox{where $a = \int_Y [\omega|_Y]^{\frac{\dim(Y)}{2}} = \int_Y [\omega|_Y]^{2k+1}\in\Z$}.$$
Since $Y$ is symplectic, $a\neq 0$. Since $\dim(M) = 2(2k+1) + 4 = 2(2k' + 1)$, $\Lambda_X = -\Lambda_Z$.  Solving the above equation, we get
$$ a = \frac{2}{\Lambda_X}\in\Z.$$
Hence $\Lambda_X = 2$ ($\Lambda_X\neq 1$ since the action is not semifree) which implies that there can only be
one isotropy submanifold between $X$ and $Z$, which is fixed by $\Z_2$.
Next, assume $\dim(Y) = 2(2k)$, where $k > 0$. We use Theorem~\ref{AB.BV} to compute
$$\int_M \left(c_1^{S^1}(M)\right)^2 = 0.$$ A similar argument as the above shows that there is only one isotropy submanifold between $X$ and $Z$, which is fixed by $\Z_2$.
Hence,  the weights of the action  at $X$ are $(1, \cdots, 1, 2)$ and the weights of the action at $Z$ are $(-1, \cdots, -1, -2)$. Then (\ref{c1m'}) gives
$c_1(M) = (n+1)[\omega].$
\end{proof}

\section{Proof of Theorem~\ref{biho}}

 In this section, we first prove Proposition~\ref{equibiho} below, then we use Proposition~\ref{equibiho} and Theorem~\ref{wc1} to  prove Theorem~\ref{biho}.

 Let us first recall  the following theorem which was proved in \cite{KO}.

\begin{theorem}\label{nonequibiho}
Let $M$ be a compact K\"ahler
manifold of complex dimension $n$. If there exists a positive element $\alpha\in H^{1, 1}(M; \Z)$ such that
$c_1(M)=(n + 1)\alpha$, then $M$ is biholomorphic to $\CP^n$; if
there exists a positive element $\alpha\in H^{1, 1}(M; \Z)$ such that
$c_1(M)=n\alpha$, then $M$ is biholomorphic to a quadratic hypersurface in $\CP^{n+1}$.
\end{theorem}

 \begin{proposition}\label{equibiho}
 Let $(M, \omega, J)$ be a compact K\"ahler manifold of complex dimension $n$, which admits a holomorphic Hamiltonian circle action. Assume that $[\omega]$ is an integral class. If $c_1(M)=(n + 1)[\omega]$, then $M$ is $S^1$-equivariantly biholomorphic to $\CP^n = \P\left(H^0(M; L)\right)$, and if $c_1(M)=n[\omega]$, then $M$ is $S^1$-equivariantly biholomorphic to a quadratic hypersurface in $\CP^{n+1}=\P\left(H^0(M; L)\right)$, where $L$ is a holomorphic line bundle over $M$ with first Chern class
 $[\omega]$ and $H^0(M; L)$ is its space of holomorphic sections.
 \end{proposition}

The proof of Proposition~\ref{equibiho} is to incorporate the circle action into the proof of Theorem~\ref{nonequibiho}.
The idea of the proof of Theorem~\ref{nonequibiho} is as follows. Since $\alpha\in H^{1, 1}(M; \Z)$, by a result of Kodaira and Spencer, there is a holomorphic line bundle $L$ over $M$ such that
  $c_1(L) = \alpha$; and since $c_1(L)$ is positive,
  $L$ is ample by a theorem of Kodaira.  Let $H^0(M; L)$ be the complex vector space of holomorphic sections of $L$ over $M$. The authors of \cite{KO} showed the following fact.
\begin{lemma}\label{dimhl}
When $c_1(M)=(n + 1) c_1(L)$, $\dim H^0(M; L) = n+1$; and when $c_1(M)= n c_1(L)$, $\dim H^0(M; L) = n+2$.
\end{lemma}
Moreover, they showed that $H^0(M; L)$ has no base points, i.e., the elements of $H^0(M; L)$
have no common zeros. Hence, if $\P\left(H^0(M; L)\right)$ is the complex projective
space defined as the set of hyperplanes through the origin in the
vector space $H^0(M; L)$, then we can define a
holomorphic mapping
\begin{equation}\label{emb}
j\colon  M\to\P\left(H^0(M; L)\right)    \quad\mbox{by}
\end{equation}
$$ j(x) = \{s\in H^0(M; L) \,|
\, s(x) = 0\} \quad\mbox{for each $x\in M$}.$$
The authors showed that
the map $j$ from $M$ to its image is biholomorphic. In the case when $c_1(M)=(n + 1) c_1(L)$,
$j(M) =\P\left(H^0(M; L)\right) =\CP^n$; and in the case when $c_1(M)= n c_1(L)$,
$j(M)$ is a quadratic hypersurface in $\P\left(H^0(M; L)\right) =\CP^{n+1}$.

Now, we make the argument work equivariantly under the assumptions of Proposition~\ref{equibiho}.
By the assumptions of Proposition~\ref{equibiho},
 $$c_1(M) = (n + 1)[\omega] \quad\mbox{or}\quad c_1(M) = n[\omega],
  \quad\mbox{where $[\omega]\in H^{1, 1}(M; \Z)$ is positive.}$$
So we have the ample line bundle $L$ as above with $c_1(L) = [\omega]$ and the results in Theorem~\ref{nonequibiho}.

We choose a Hermitian structure $h$ on $L$ and a Hermitian
connection $\bigtriangledown$  such that its curvature form is
$\omega$. Then  $\left(L, h,
\bigtriangledown\right)$ is a (pre)quantum line bundle over $M$.

Let $X_M$ be the vector field generated by the circle action on $M$,
and let $\Tilde{X_{M}}$ be the horizontal lift of $X_M$ to $TL$.
Then
$$X_L = \Tilde{X_{M}} + \phi\,\frac{\partial}{\partial\theta}$$
gives an action of the Lie algebra of $S^1$ on $L$, where
$\frac{\partial}{\partial\theta}$ is the vector field on $L$
generated by the fiberwise multiplication by $e^{i\theta}$.

Let $X, Y, Z$ be the fixed components of the circle action such that $\phi(X) < \phi(Y) < \phi(Z)$.
Since $[\omega]$ is an integral class, Lemma~\ref{extension} implies that
\begin{equation}\label{y-x, z-y}
\phi(Y) - \phi(X)\in \N\quad\mbox{and}\quad \phi(Z)-\phi(Y)\in \N.
\end{equation}
 Since the moment
map is defined up to translation by a constant, (\ref{y-x, z-y})
implies that we may assume the moment map values of the fixed point
set components are in the integral lattice of $\R$. By \cite[Example 6.10]{GGK},
the above action of the Lie algebra of $S^1$  on $L$
integrates to an $S^1$ action on $L$ which is compatible with the
$S^1$ action on $M$. Then $S^1$ acts on $H^0(M; L)$  which is given
by
\begin{equation}\label{s1s}
\left(\lambda\cdot s\right)(x) = \lambda\cdot s(\lambda^{-1}\cdot
x), \quad\mbox{where $\lambda\in S^1$ and $s\in H^0(M; L)$};
\end{equation}
 or infinitesimally given by
$$X\cdot s = - \triangledown_{X_M} s +  \sqrt{-1} \phi\cdot s, \quad\mbox{where $X\in$ Lie$(S^1)$
and $s\in H^0(M; L)$}.$$

The $S^1$ action on $H^0(M; L)$ induces an $S^1$ action on
$\P\left(H^0(M; L)\right)$.
\begin{lemma}\label{equiv}
Under the assumptions of Proposition~\ref{equibiho}, the holomorphic map $j$ in (\ref{emb}) can be made to be $S^1$-equivariant.
\end{lemma}

\begin{proof}
The above analysis shows that $S^1$ acts on $H^0(M; L)$ and on $\P\left(H^0(M; L)\right)$.
For any $g\in S^1$,  and $s\in H^0(M; L)$, by (\ref{s1s}), we have
$$(g\cdot s) (x) = g\cdot s(g^{-1}\cdot x).$$
So
$$j(g\cdot x) = \{s\in H^0(M; L) \,| \, s(g\cdot x) = 0\} \,\,\mbox{for each $x\in M$,  and}$$
$$g\cdot j(x) = \{g\cdot s\in H^0(M; L) \,| \, g\cdot s(x) = 0\}= \{g\cdot s\in H^0(M; L) \,| \, g\cdot s(g^{-1}g \cdot x) = 0\}$$
$$=\{g\cdot s\in H^0(M; L) \,| \, (g\cdot s) (g\cdot x) = 0\}
\,\,\mbox{for each $x\in M$}.$$ Hence $$j(g\cdot x) =g\cdot j(x)
.$$
\end{proof}

Proposition~\ref{equibiho} follows from Theorem~\ref{nonequibiho} and Lemma~\ref{equiv}.

\begin{remark}\label{kemb}
The embedding $j$ in (\ref{emb}) may not be a K\"ahler embedding.
 In the case when $M$ is a compact homogeneous  K\"ahler manifold which
  admits a homogeneous very ample quantum line bundle $L$,
 for  example, when $M$ is $\CP^n$ or $\Gt_2(\R^{n+2})$, the
 embedding is K\"ahler \cite{{BS}, {CGR}}.
\end{remark}

\begin{proof}[Proof of Theorem~\ref{biho}]
 By Lemma~\ref{1dim}, we can rescale the  K\"ahler form $\omega$ so that $[\omega]$ is primitive integral.

 By quotienting out a finite subgroup of $S^1$ that acts trivially, we may assume that the $S^1$ action is effective.
 Then by Theorem~\ref{wc1}, $c_1(M)=(n + 1)[\omega]$ with $n\geq 2$, or $c_1(M)= n [\omega]$ with $n\geq 3$.

 By Proposition~\ref{equibiho},
 $M$ is equivariantly biholomorphic to $\P\left(H^0(M; L)\right)=\CP^n$
 or to a quadratic hypersurface in $\P\left(H^0(M; L)\right)=\CP^{n+1}$ through the equivariant holomorphic map $j$ in (\ref{emb}).  In particular, the circle action on $j(M)$ has $3$ fixed components biholomorphic to those on $M$.

 Since the $S^1$ action on the complex vector space $H^0(M; L)$ is Hamiltonian,
 the induced $S^1$ actions on $\P\left(H^0(M; L)\right)$ and on the invariant complex manifold $j(M)$ are Hamiltonian.
 In the case when $j(M)$ is a quadratic hypersurface in $\CP^{n+1}$, $j(M)$ can be identified with
 $\Tilde G_2\left(\R^{n+2}\right)$.
\end{proof}

\section{Proof of Theorem~\ref{symp}}

 With the help of Theorem~\ref{biho}, we can prove Theorem~\ref{symp} as follows.

\begin{proof}[Proof of Theorem~\ref{symp}]
By Theorem~\ref{biho}, there is an $S^1$-equivariant biholomorphism
$$f\colon (M, \omega, J) \to  (M', \omega', J'),$$
where $(M', \omega', J')$ represents the K\"ahler manifold $\CP^n$ with $n\geq2$ or $\Tilde G_2\left(\R^{n+2}\right)$ with $n\geq 3$, with the standard structures.
By Lemma~\ref{1dim},   $H^2(M; \R)=\R$.  So by rescaling $\omega$, we may assume that $\omega$ and $f^*\omega'$ represent  the same cohomology class.  Consider the family of forms $\omega_t = (1-t)\omega + t f^*\omega'$, where $t\in [0,1]$.
Each $\omega_t$ is clearly closed.  Each $\omega_t$ is also nondegenerate:  for any point $x\in M$, suppose $X\in T_x M$ is such that $\omega_t (X, Y) = 0$ for all $Y\in T_x M$. In particular, if we take $Y = J X$,
then we get $\omega_t (X, J X) = 0$. Using the facts $\omega(X, JX)\geq 0$, $f_* (J X) = J' f_* X$, and $\omega'(f_*X, J'f_* X)\geq 0$, we get $X= 0$.  So $\omega_t$ is a family of symplectic forms. It represents the same cohomology class and it is $S^1$-invariant. By Moser's theorem \cite{Mo} (Moser's argument without the presence of a group action can be adapted to the case when a compact group acts by choosing invariant objects), there is an $S^1$-equivariant isotopy $\Phi_t$ such that $\Phi_t^* \omega_t = \omega$; in particular, $\Phi_1^* f^*\omega' = \omega$.
\end{proof}

\end{document}